\documentclass[reqno,11pt]{amsart}
\usepackage{amsmath,amssymb,amsthm, verbatim}
\usepackage{url}
\usepackage[usenames, dvipsnames]{color}

\usepackage[letterpaper,hmargin=1in,vmargin=1in]{geometry}

\usepackage{graphicx}
\usepackage{enumerate,pinlabel}
\usepackage{mathrsfs,graphicx,url}
\usepackage[usenames,dvipsnames]{xcolor}
\usepackage[colorlinks=true,linkcolor=Black,citecolor=Black, urlcolor=Black,pagebackref]{hyperref}

\numberwithin{equation}{section}

\theoremstyle{plain}
\newtheorem{theorem}{Theorem}[section]

\newtheorem*{theorem*}{Theorem}
\newtheorem*{lemma*}{Lemma}
\newtheorem{lemma}[theorem]{Lemma}
\newtheorem{proposition}[theorem]{Proposition}
\newtheorem{corollary}[theorem]{Corollary}

\theoremstyle{definition}

\theoremstyle{remark}
\newtheorem{remark}[theorem]{Remark}

\newcommand{\ep}{\varepsilon}

\newcommand{\R}{\mathbb{R}}

\newcommand{\C}{\mathbb{C}}
\newcommand{\D}{\mathcal{D}}

\newcommand\supp{\mathop{\rm supp}}

\newcommand\real{\mathop{\rm Re}}
\newcommand\imag{\mathop{\rm Im}}
\newcommand\loc{\operatorname{loc}}

\newcommand\Res{\mathop{\rm Res}}
\newcommand*{\defeq}{\mathrel{\vcenter{\baselineskip0.5ex \lineskiplimit0pt

                     \hbox{\scriptsize.}\hbox{\scriptsize.}}}%
                     =}
                     
\newcommand*{\qefed}{=\mathrel{\vcenter{\baselineskip0.5ex \lineskiplimit0pt

                     \hbox{\scriptsize.}\hbox{\scriptsize.}}}}

\title[Measure potentials on the line]{Semiclassical estimates for measure potentials on the real line}
\begin{document}
\author{Andr\'es Larra\'in-Hubach}
\address{Department of Mathematics, University of Dayton, Dayton, OH 45469-2316, USA}
\email{alarrainhubach1@udayton.edu}
\author{Jacob Shapiro}
\address{Department of Mathematics, University of Dayton, Dayton, OH 45469-2316, USA}
\email{jshapiro1@udayton.edu}

\begin{abstract}
 We prove an explicit weighted estimate for the semiclassical Schr\"odinger operator $P = - h^2 \partial^2_x + V(x;h)$ on $L^2(\R)$, with $V(x;h)$ a finite signed measure, and where $h >0$ is the semiclassical parameter. The proof is a one dimensional instance of the spherical energy method, which has been used to prove Carleman estimates in higher dimensions and in more complicated geometries. The novelty of our result is that the potential need not be absolutely continuous with respect to Lebesgue measure. Two consequences of the weighted estimate are the absence of positive eigenvalues for $P$, and a limiting absorption resolvent estimate with sharp $h$-dependence. The resolvent estimate implies exponential time-decay of the local energy for solutions to the corresponding wave equation with a compactly supported measure potential, provided there are no negative eigenvalues and no zero resonance, and provided the initial data have compact support. 
\end{abstract}

\maketitle


\section{Introduction and statement of results} \label{introduction section}

The goal of this note is to study the spectral and scattering theory for the one dimensional semiclassical Schr\"odinger operator,
\begin{equation} \label{P}
 P = P(h) \defeq - h^2 \partial_x^2 +V(x;h) : L^2(\R) \to L^2(\R), \qquad h > 0,
\end{equation}
with potential $V = V(x;h)$ a real, finite signed Borel measure on $\R$, which may depend on the semiclassical parameter $h$. Here and below, $L^2(\R)$ is the usual Hilbert space of equivalence classes of functions $u : \R \to \C$ which are measurable with respect to the Lebesgue sigma algebra on $\R$, and for which $\int_\R |u|^2dx <\infty$, where $dx$ denotes Lebesgue measure. 

Self-adjointness of singular Sturm-Liouville operators encompassing \eqref{P} was systematically addressed in earlier works \cite{hrmy01, hrmy12, egnt13}. With the objective of being self contained, we proceed in elementary fashion to specify the domain  $\mathcal{D}$ of $P$ as a certain dense subspace of $L^2(\R)$ which is contained in the Sobolev space $H^1(\R)$. Recall each $u \in H^1(\R)$ has a (unique) continuous, bounded representative, which we denote by $u_c$. Thus, for $u \in H^1(\R)$, we prescribe the product of $u$ and $V$ to be the complex Borel measure $u_c V$, and define the expression
\begin{equation} \label{apply P}
Pu \defeq -h^2 \partial^2_x u + u_c V
\end{equation}
 in the sense of distributions on $\R$. Using the calculus of functions of bounded variation (which we review in Section \ref{BV review section}), we show in Section \ref{self adjointness section} that (for all $h > 0$) $P$ is self-adjoint with respect to
\begin{equation} \label{D}
    \D \defeq \{u \in H^1(\R) : u' \in L^\infty(\R), \text{ and } Pu  \in L^2(\R) \}.
\end{equation}

The main result of this note, whose proof appears in Section \ref{measure section}, is the following weighted estimate on $L^2(\R)$:
\begin{theorem} \label{meas Carleman est thm}
Fix $\delta > 0$. For all $E = E(h) > 0$ (which may depend on $h$), $\ep \in [0,1]$, $h > 0$, and $u \in \mathcal{D}$ with $(|x| + 1)^{(1 + \delta)/2} (P - E\pm i\ep) u \in L^2(\R)$, 
\begin{equation} \label{meas Carleman est} 
\begin{split}
  \int_{\R} (|x| + 1)^{-1-\delta}(E|u|^2 &+ |hu'|^2)dx  \\
  &\le  C(V, E, h , \delta ) \int_{\R} (|x| + 1)^{1+\delta} |(P- E \pm i\ep) u |^2 dx.
  \end{split}
\end{equation}
Here, 
\begin{gather}
C(V,E,h, \delta) \defeq e^{C_1(V,E,h,\delta)} \Big( \big(\frac{4}{h^2} + \frac{2}{h} + \frac{1}{Eh^2} \big(2 + 2E + \frac{\|V\|^2}{h^2} \big)^2  e^{C_1(V,E,h,\delta)}  \Big), \label{C} \\
C_1(V, E, h, \delta) \defeq 2 \delta^{-1} + E^{-1/2} h^{-1} \|V\|, \label{C1}
\end{gather}
and $\| V\| \defeq |V|(\R)$, with$|V|$ the total variation of $V$, defined by
\begin{equation*}
|V| = V^+ + V^{-},
\end{equation*}
where $\{V^+, V^-\}$ is the Jordan decomposition of $V$ (see, e.g., \cite[Theorem 3.4]{fo}).
\end{theorem}
\begin{remark}
In the case that $E > 0$ is fixed independent of $h$, and we restrict $h \in (0,1]$, the constant \eqref{C} is bounded from above by the more succinct expression
\begin{equation} \label{simplified up bd}
\exp (\tilde{C}(E,\delta)(1 + \|V\|)/h), \qquad \text{for some $\tilde{C}(E, \delta) > 0$ depending on $E$ and $\delta$.}
\end{equation}
\end{remark}

Two consequences of Theorem \ref{meas Carleman est thm} are the absence of positive eigenvalues of $P$, and a weighted limiting absorption resolvent bound.

\begin{corollary}
The operator $P$ on $L^2(\R)$, given by \eqref{apply P} and equipped with domain \eqref{D}, has no positive eigenvalues.
\end{corollary}

\begin{corollary}\label{meas resolv est thm}
Fix  $\delta > 0$. Then for all $ E = E(h) > 0$ (which may depend on $h$), $\ep \in (0,1]$, and $h > 0$, 
\begin{equation} \label{exp est meas}
 \|(|x| + 1)^{-\frac{1+\delta}{2}} (P(h) - E \pm i\varepsilon)^{-1} (|x| + 1)^{-\frac{1+\delta}{2}}\|_{L^2(\mathbb R) \to L^2(\mathbb R)} \le \Big( \frac{C(V, E, h ,\delta)}{E} \Big)^{1/2} ,
\end{equation}
where $C(V,E,h, \delta)$ is as given in \eqref{C}. 
\end{corollary}

\begin{remark}
It is well known that $V \in L^1(\R; \R)$ implies absence of positive eigenvalues. Furthermore, absence of positive eigenvalues was proved, by different methods, for locally $H^{-1}$ potentials with $L^1$-type decay, see \cite[Theorem 1.9]{lsw22}. This class includes finite signed measures as a special case. Recall also the celebrated von Neumann-Wigner potential $W$ \cite[Section XIII.13]{rs78}, which obeys
\begin{equation*}
W(x) = -\frac{8 \sin(2 |x|)} {|x|} + O(|x|^{-2}), \qquad \text{as $|x| \to \infty$,}
\end{equation*}
and has an eigenvalue at $E = 1$.
\end{remark}

\begin{remark}
For $E > 0$ fixed and $h \in (0,1]$, the right side of \eqref{exp est meas} is bounded from above by an expression of the form \eqref{simplified up bd}. In higher dimensions, resolvent upper bounds like \eqref{simplified up bd} are usually proved by first establishing a Carleman estimate, which is similar to \eqref{meas Carleman est} but involves an additional weight of the form $e^{\varphi/h}$, where $\varphi$ is a suitable phase function (see e.g., \cite[Theorem 2.2]{cv} and \cite[Lemma 2.2]{da}). However, our proof of Theorem \ref{meas Carleman est thm} in Section \ref{measure section} shows that in one dimension it is not necessary to use a phase. 
\end{remark}

When $V$ is compactly supported, we prove a simpler weighted estimate away from the support of $V$, which yields an improvement to \eqref{exp est meas} for $h$ small.

\begin{theorem}[exterior estimate] \label{ext est meas thm}
Fix $\delta > 0$ and $E > 0$. Suppose $V$ is supported in $[-R_0, R_0]$ (independent of $h$) for some $R_0 > 0$. Set $h_0 = 2^{-1} \delta^{-1} (1 + R_0)^{-\delta}$. There exist $C$, depending only on $R_0, E$, and $\delta$ (see \eqref{penult meas ext est} below), so that for all $\ep \in (0,1]$ and $h \in (0,h_0]$,
\begin{equation} \label{ext est meas}
\|(|x| + 1)^{-\frac{1+\delta}{2}} \mathbf{1}_{> R_0} (P - E \pm i\varepsilon)^{-1} \mathbf{1}_{> R_0} (|x| + 1)^{-\frac{1+\delta}{2}}\|_{L^2(\mathbb R) \to L^2(\mathbb R)} \le \frac{C}{h}.
\end{equation}
Here, $\mathbf{1}_{> R_0}$ denotes the characteristic function of $\{|x| > R_0\}$. 
\end{theorem}

The $h$-dependencies in \eqref{exp est meas} and \eqref{ext est meas} are sharp in general, and were proved previously for $V \in L^1(\R ; \R)$  \cite{dash20}. Thus the novelty of this work is that Theorem \ref{meas Carleman est thm} implies optimal semiclassical resolvent bounds for potentials in one dimension which may not be absolutely continuous with respect to Lebesgue measure.

When $V$ is smooth, the exponential bound \eqref{exp est meas} (with different constants) was first proved by Burq \cite{bu98, bu02}, who also considered higher dimensions and more general operators. Further proofs and generalizations can be found in \cite{cv, da, rota15, ddh,sh19, ga19, gash22b, ob23}. In dimension $n > 1$, current results require at least $V \in L_{\text{loc}}^\infty(\R^n \setminus \{0\})$, with sufficient decay toward infinity, to obtain a semiclassical resolvent estimate. And often, only weaker versions of \eqref{exp est meas} are known \cite{klvo19, vo19, vo20a, vo20b, vo20c, vo21, sh20, vo22, gash22a, dgs23, sh24}, with $e^{C/h}$ replaced by $e^{C/h^\ell}$ for some $\ell >1$.

When $V$ is smooth, the improvement \eqref{ext est meas} away from the support of $V$ was first proved by Cardoso and Vodev \cite{cv}, refining earlier work of Burq \cite{bu02}, and again analogous results hold in many settings \cite{cv, da, rota15, ddh, sh19, gash22b, ob23}. When the dimension $n > 1$, Datchev and Jin \cite{daji20} showed the cutoff ${\bf 1}_{|x|>R_0}$ may need to be replaced by ${\bf 1}_{|x|>R}$ with $R \gg R_0$, even when $V \in C_0^\infty(\mathbb R^n)$.

To prove Theorem \ref{meas Carleman est thm}, we employ a positive commutator-style argument in the context of the so-called spherical energy method. This strategy has long been used to prove Carleman and related estimates \cite{cv, da, klvo19, dash20, gash22b, ob23}. In fact, as we work in one dimension, it suffices to use the pointwise energy
\begin{equation} \label{F}
\begin{gathered}
    F(x) = F_\pm[u](x) \defeq |hu'(x)|^2 + E|u(x)|^2, \\ 
      u \in \mathcal{D} \text{ such that } (|x| + 1)^{(1 + \delta)/2} (P(h)- E \pm i\varepsilon)u \in L^2(\R).
       \end{gathered}
\end{equation}

The goal is to construct a suitable weight $w(x)$ having locally bounded variation, so that, roughly speaking, the distributional derivative of $wF$ is bounded from above by a term involving \\ $2w \real ((P - E \pm i\ep)u \overline{u}')$, and bounded from below by $w(E|u|^2 + |hu'|^2)$ (see \eqref{lwr bd d w eta F} for the precise estimate). To attain the upper bound, $V$ needs to be reintroduced after differentiation of $wF$, at the cost of a perturbation term (see \eqref{meas eta dwF}, and note $V$ does not appear in $F$ in the first place because its distributional derivative may be irregular). This perturbation can be controlled, yielding the desired the lower bound, by designing $w$ appropriately. In particular, since $V$ may have discrete part $V_d$ (i.e., countably many point masses which are absolutely summable), we use a family of weights $w_\eta(x)$ depending on a parameter $\eta > 0$. Each $w_\eta$ controls a certain Gaussian approximation of $|V_d|$ (see \eqref{smoothed discrete part}). We then show that the needed estimates hold uniformly as $\eta \to 0^+$. 

When $V$ is compactly supported, it is well known that Corollary \ref{meas resolv est thm} is related to the distribution of scattering resonances for the operator $-\partial^2_x + V$. As in \cite{sz91}, we define the resonances of $-\partial^2_x + V$ as the poles of the cutoff resolvent
\begin{equation} \label{cutoff resolv}
    \chi (-\partial^2_x + V - \lambda^2)^{-1} \chi : L^2(\R) \to \mathcal{D}, \qquad \chi \in C^\infty_0(\R ; [0,1]), \chi \equiv 1 \text{ on } \supp V,
\end{equation}
which continues meromorphically from $\imag \lambda \gg 1$ to the complex plane. In Section \ref{unif resolv est section}, we combine \eqref{exp est meas} with a resolvent identity argument of Vodev \cite[Theorem 1.5]{vo14} to show
\begin{theorem} \label{unif resolv est thm}
Suppose $V$ is a finite signed Borel measure on $\R$, which is supported in $[-R_0, R_0]$ for some $R_0 >0$.  Fix  $\chi \in C_0^\infty(\mathbb R; [0,1])$ such that $\chi =1$ near $[-R_0, R_0]$, and fix $\lambda_0 > 0$. There exist $C, \ep_0 > 0$ so that for all $|\real \lambda| \ge \lambda_0$, and $|\imag \lambda| \le \ep_0$,
 \begin{equation} \label{unif resolv est}
      \|\chi (-\partial^2_x + V - \lambda^2)^{-1} \chi\|_{L^2 \to H^k} \le C|\real \lambda |^{k-1}, \qquad k = 0, \, 1 \,(H^0 \defeq L^2(\R)),
 \end{equation}
 and 
 \begin{equation} \label{unif resolv est D}
      \|\chi (-\partial^2_x + V - \lambda^2)^{-1} \chi\|_{L^2 \to \mathcal{D}} \le C(|\real \lambda | + 1),
 \end{equation}
 where $\mathcal{D}$ is equipped with the graph norm $\| u \|_{\mathcal{D}} \defeq (\|(-\partial^2_x + V)u\|^2_{L^2} + \|u\|^2_{L^2})^{1/2}$.
\end{theorem}

The existence of resonance free regions below the real axis is a long-studied problem: \cite{ha82, zw87, hi99} treat the cases $V \in L^\infty_{\text{comp}}(\R)$, $V \in L_{\text{comp}}^1(\R)$, and $V$ exponentially decaying, respectively. More recent articles \cite{sa16, dmw22, dama22} describe the distribution of resonances for thin barriers in the semiclassical regime. To the authors' knowledge, Theorem \ref{unif resolv est thm} is the first demonstration of a resonance free strip for a class of potentials in one dimension that can have singularly continuous part.

Estimates such as \eqref{unif resolv est} and \eqref{unif resolv est D} yield regularity and decay results for operators involving the measure $V$. As an illustration, in Section \ref{wave decay section} we show how \eqref{unif resolv est} implies an exponential local energy decay rate, modulo negative eigenvalues and a possible zero-resonance, for the associated wave equation, 
\begin{equation} \label{wave eqn}
\begin{cases}
(\partial_t^2 - \partial_x^2  + V(x))w(x,t) = 0, \qquad (x,t) \in \R \times (0, \infty), \\
w(x,0) = w_0(x),  \\
\partial_t w(x,0) = w_1(x), \\
\supp w_0, \, \supp w_1 \subseteq (-R,R), \qquad R > 0. 
\end{cases} 
\end{equation}
See Theorem \ref{LED thm} for a precise statement. Similar wave decay was previously established for \\$V \in L_{\text{comp}}^\infty(\R ; \R)$ (\cite{tz} and \cite[Theorem 2.9]{dz}). We also mention that exterior estimates like \eqref{ext est meas} have application to integrated wave decay \cite[Lemma 5]{daji20}. 



There is an extensive literature on second order operators whose coefficients are singular. Thus, we will not attempt to give a comprehensive review here. For the one dimensional case,  we point the reader to \cite{sash99, hrmy01,  amre05, hrmy12, ecte13, egnt13, ekmt14}, which develop the Sturm-Liouville theory for such operators, and investigate topics including boundary conditions, self-adjoint extensions, and inverse spectral theory. The research monograph \cite{smq} gives a comprehensive treatment of point interactions in three and fewer dimensions. Higher dimensional studies include \cite{beks92, gal19}.

\section{Review of BV} \label{BV review section}

To keep the notation concise, for the rest of the article, we use ``prime" notation to denote differentiation with respect to $x$, e.g., $u' \defeq \partial_x u$.

In Section \ref{BV review section}, we review the basics of functions of bounded variation (BV), and collect four well-known Propositions concerning their calculus. This material is relied upon frequently in later Sections. We give the proof of Proposition \ref{ftc bv prop}, while proofs of Propostions \ref{ibp bv prop}, \ref{prod rule bv prop}, and \ref{chain rule bv prop} may be found in \cite[Appendix B]{dash22}.

Let $f : \R \to \C$ be a function of locally bounded variation. For all $x \in \R$, put 
\begin{equation} \label{LRA}
 f^L(x) \defeq \lim_{\delta \to 0^+}f(x-\delta), \qquad  f^R(x) \defeq \lim_{\delta \to 0^+}f(x+\delta), \qquad f^A(x) \defeq (f^L(x) + f^R(x))/2,
\end{equation}
where the limits exist because both the real and imaginary parts of $f$ are a difference of two increasing functions. Recall that $f$ is differentiable Lebesgue almost everywhere, so $f (x)= f^L(x) = f^R(x) = f^A(x)$ for almost all $x \in \R$.

We may decompose $f$ as 
\begin{equation} \label{decompose f}
f = f_{r, +} - f_{r, -} + i( f_{i,+} - f_{i,-}),
\end{equation}
where the $f_{\sigma,\pm}$, $\sigma \in \{r, i\}$, are increasing functions on $\R$. Each $f^R_{\sigma,\pm}$ uniquely determines a regular Borel measure $\mu_{\sigma,\pm}$ on $\R$ satisfying $\mu_{\sigma, \pm}(x_1, x_2] = f^R_{\sigma, \pm}(x_2) -  f^R_{\sigma, \pm}(x_1)$, see \cite[Theorem 1.16]{fo}. We put
\begin{equation} \label{df}
    df \defeq \mu_{r, +} - \mu_{r, -} + i( \mu_{i,+} - \mu_{i,-}),
\end{equation}
 which is a complex measure when restricted to any bounded Borel subset. For any $a  < b$,
 
 \begin{equation} \label{ftc}
 \begin{gathered}
  \int_{(a,b]}df = f^R(b) - f^R(a),\\
  \int_{(a,b)}df = f^L(b) - f^R(a).
  \end{gathered}
 \end{equation}

\begin{proposition}[integration by parts] \label{ibp bv prop}
Let $f: \R \to \C$ have locally BV. For any $a < b$, and any continuous $\varphi$ with $\varphi'$ piecewise continuous, 
\begin{equation} \label{Folland IBP}
 \int_{(a,b]} \varphi df + \int_{(a,b]} \varphi' fdx = f^R(b)\varphi(b) - f^R(a)\varphi(a).
\end{equation}
\end{proposition}

\begin{proposition}[funadmental theorem of calculus] \label{ftc bv prop}
Let $\mu_{\sigma, \pm}$, $\sigma \in \{r,\,i\}$ be positive Borel measures on $\R$ which are finite on all bounded Borel subsets of $\R$. Suppose $u \in  \mathcal{D}'(\R)$ has distributional derivative equal to $\mu = \mu_{r, +} - \mu_{r, -} + i( \mu_{i,+} - \mu_{i,-})$. (For example, this will hold for $u \in \mathcal{D}$, with $\mu = u_c V + gdx$ for some $g \in L^2(\R)$.) Then $u$ is a function of locally BV; For any $a \in \R$, $u$ differs by a constant from the right continuous, locally BV function
\begin{equation} \label{funct from meas}
f_\mu(x) \defeq \begin{cases}
\int_{[a, x]} d\mu & x \ge a, \\
-\int_{(x,a)} d\mu & x < a. 
\end{cases}
\end{equation}
\end{proposition}
\begin{proof}
We need to show that the function \eqref{funct from meas} has distributional derivative $\mu$. First, it is straightforward to check that $f_\mu(x_2) - f_\mu(x_1) = \mu(x_1, x_2]$ for all $x_1 < x_2$. Hence $df_\mu = \mu$. Then \eqref{Folland IBP} implies
\begin{equation} \label{distributional identity}
-\int_{\R} \varphi' f_\mu dx = \int_{\R} \varphi d\mu, \qquad   \varphi \in C^\infty_0(\R),
\end{equation}
where all boundary terms vanish, and the right side of \eqref{distributional identity} is finite, due to the the compact support of $\varphi$. \\
\end{proof}

\begin{proposition}[product rule] \label{prod rule bv prop}
Let $f, \, g : \R \to \C$ be functions of locally bounded variation. Then
\begin{equation}\label{e:prod}
 d(fg) = f^A dg + g^A df
\end{equation}
as measures on a bounded Borel subset of $\R$.
\end{proposition}


\begin{proposition}[chain rule] \label{chain rule bv prop}
Let $f : \R \to \R$ be continuous and have locally bounded variation. Then, as measures on a bounded Borel set of $\R$,
\begin{equation} \label{chain rule continuous}
    d(e^f) = e^{f} df.
\end{equation}
\end{proposition}

\section{Self-adjointness for $V$ a measure} \label{self adjointness section}

The goal of this section is to use the tools of Section \ref{BV review section} to show $(P, \mathcal{D})$ is self-adjoint on $L^2(\R)$, where $\mathcal{D}$ is given by \eqref{D}. This strategy sets the stage for several steps in the proof of Theorem \ref{meas Carleman est thm} in Section \ref{measure section}. We demonstrate that $(P, \mathcal{D})$ is merely the self-adjoint operator naturally associated to the quadratic form
\begin{equation} \label{quad form}
 q(u,v) \defeq \int h^2 \overline{u}' v'dx + \int \overline{u_c} v_c V, \qquad u, \, v \in H^1(\R). 
\end{equation}
As mentioned in Section \ref{introduction section}, self-adjointness was addressed in greater generality elsewhere \cite{hrmy01,hrmy12, egnt13}.
 \begin{lemma} \label{self adjointness lemma}
Let $V = V(x;h)$ be a real, finite signed Borel measure on $\R$. Then $\mathcal{D}$ specified by \eqref{D} is dense in $L^2(\R)$. The operator $P : L^2(\R) \to L^2(\R)$ given by \eqref{apply P} with domain $\mathcal{D}$ is self-adjoint.
\end{lemma}

\begin{proof}
Throughout the proof, we work with $u \in L^2(\R)$ that have locally integrable distributional derivative $u'$. Such $u$ have a (unique, locally absolutely) continuous representative $u_c$, hence we can define the product $u_cV$ as a distribution on $\R$, since it is a complex measure when restricted to bounded Borel subsets of $\R$.

Let $\mathcal D_\textrm{max} \supseteq \mathcal{D}$ be the set of all $u \in L^2(\R)$ such that $u' \in L^1_\textrm{loc}(\R)$ and $Pu \defeq -h^2 u'' + u_c V \in L^2(\R)$. By Proposition \ref{ftc bv prop}, for any $u \in \mathcal D_{\textrm{max}}$, we may fix a representative $u_{\text{bv}}'$ for $u'$ that has locally bounded variation. If necessary, we redefine $u_{\text{bv}}'$ on a set of Lebesgue measure zero so that $u_{\text{bv}}'(x) = (u_{\text{bv}}')^A(x)$ for all $x \in \R$ (this updated $u'_{\text{bv}}$ still has locally BV). 

In the computations to follow, we always work with the representatives $u_c$ and $u_{\text{bv}}'$, but we drop subscripts to keep notation concise. This convention ensures that expressions like $u'V$ are well defined as complex Borel measures (on bounded Borel subsets), since locally BV functions are Borel measurable. It also simplifies some calculations that involve \eqref{ftc} or \eqref{e:prod}. 

Our first step is to prove $\mathcal D_\textrm{max} \subseteq \mathcal D$. Since the reverse containment is trivial, we will conclude $\mathcal D_\textrm{max} = \mathcal D$. Indeed, for $u \in \mathcal D_\textrm{max}$ and any $a>0$,
\begin{equation} \label{bound L2 u prime}
\begin{split}
\int_{(-a,a)} |u'|^2dx &= \int_{(-a, a)} u' d(\overline{u}) \\
& = \int_{(-a,a)} d(u' \overline{u}) - \overline{u} d(u') \\
&=  (u'\bar u)(a) - (u'\bar u)(-a)  + h^{-2} \int_{(-a, a)}  \overline{u} Pu dx - h^{-2} \int_{(-a,a)} \overline{u} uV \\
& \le 2\sup_{(-a,a)}|u'|\sup_{(-a,a)} |u| + h^{-2}\|V\| \sup_{(-a,a)}|u|^2 + h^{-2}\|Pu\|_{L^2}\|u\|_{L^2},
\end{split}
\end{equation}
where $\|V\| \defeq |V|(\R)$, with $|V|$ the total variation of $V$. The second line of \eqref{bound L2 u prime} follows from \eqref{e:prod} and $u = u^A$, $u' = (u')^A$; the third line follows from the fact that $-h^2 d(u') = Pu - uV$ as Borel measures, which is a consequence of \eqref{distributional identity}.

Since $u$ is locally absolutely continuous,
\begin{equation}
\sup_{(-a,a)}|u|^2 =\sup_{x \in (-a,a)}  \Big(|u(0)|^2 + 2 \real \int_0^x u' \overline{u}  dx \Big)  \le |u(0)|^2 + 2\Big(\int_{(-a,a)}|u'|^2 dx\Big)^{1/2}\|u\|_{L^2}.
\end{equation}
Furthermore, if $x \in (0, a)$, then by \eqref{ftc}, \eqref{e:prod} and $u' = (u')^A$,
\begin{equation*}
\begin{split}
(u' &\overline{u}')(x) \\
&= (u' \overline{u}')(0) + \int_{(0,x]} d (u' \overline{u}' )  \\
&= |u'|^2(0) - 2 h^{-2} \real \Big( \int_{(0, x]} \overline{u}' Pud x -  \int_{(0, x]} \overline{u}'  uV \Big) \\
&\le |u'|^2(0) + 2 h^{-2}\|V\|\sup_{(-a,a)}|u|  \sup_{(-a,a)}|u'| +  2h^{-2}\|Pu\|_{L^2}\Big(\int_{(-a, a)} |u'|^2 dx \Big)^{1/2},
\end{split}
\end{equation*}
while if $x \in (-a,0)$, we similarly find
\begin{equation*}
\begin{split}
 (u'& \overline{u}')(x)\\
 &= (u' \overline{u}')(0) -
\int_{(x,0]} d (u' \overline{u}' ) \\
&\le |u'|^2(0) + 2 h^{-2}\|V\|\sup_{(-a,a)}|u|  \sup_{(-a,a)}|u'| +  2h^{-2}\|Pu\|_{L^2}\Big(\int_{(-a, a)} |u'|^2 dx \Big)^{1/2}.
\end{split}
\end{equation*}


We thus arrive at a system of inequalities of the form $x^2 \le 2yz+Ay^2 + B$, $y^2 \le C + Dx$, $z^2 \le E + Fyz + Gx$, where $x \defeq (\int_{(-a,a)}|u'|^2 dx)^{1/2}$, $y \defeq \sup_{(-a,a)}|u|$, and $z \defeq \sup_{(-a,a)}|u'|$. After using the second inequality to eliminate $y$, we obtain a system in $x$ and $z$ with quadratic left hand sides and subquadratic right hand sides. Hence $x$, $y$, and $z$ are each bounded in terms of $A, B, \dots, G$. Letting $a \to \infty$, we conclude that $u' \in L^2(\R)$ and $u, \, u' \in L^\infty(\R)$. Hence  $\mathcal D_\textrm{max} \subseteq \mathcal D$ as desired.

Next, we equip $P$ with the domain $\mathcal D_\textrm{max} = \mathcal D$, and show that $P$ is symmetric. Let $u, v \in \mathcal{D}$, and take $\{\varphi_k\}_{k=1}^\infty \subseteq C^\infty_0(\R)$ converging to $v$ in $H^1(\R)$. Using the distributional definition of $Pu$, 
\begin{equation} \label{prod to form}
\begin{split}
\langle Pu, v \rangle_{L^2} &= \lim_{k \to \infty} \langle Pu, \varphi_k \rangle_{L^2} \\
&=\lim_{k \to \infty} \int \overline{u} (-h^2 \varphi_k'')dx + \int  \overline{u} \varphi_k V\\
&=\lim_{k \to \infty} \int h^2\overline{u}' \varphi_k' dx + \int  \overline{u} \varphi_k V\\
&=  \int h^2 \overline{u}' v'dx + \int \overline{u} v V,
\end{split}
\end{equation}
where the last equal sign follows since $|V|$ is finite and $\|w\|^2_{L^\infty} \le \|w\|_{L^2}\|w'\|_{L^2}$ for any $w \in H^1(\R)$. Approximating $u \in H^1(\R)$ by $C^\infty_0(\R)$-functions, we similarly have $\langle u, Pv \rangle_{L^2} =  \int h^2 \overline{u}' v'dx + \int  \overline{u} v V$. Thus $P$ is symmetric. 

The last step is to establish that $(P, \mathcal{D})$ is densely defined and $P^* \subseteq P$. For this, define on $H^1(\R)$ the quadratic form \eqref{quad form}. Since, for any $\gamma > 0$,
\begin{equation} \label{bdd below}
\begin{split}
 \left| \int V |u|^2\right| &\le \|V\| \|u\|_{L^\infty}^2 \\
&\le \|V\| \|u\|_{L^2} \|u'\|_{L^2} \\
& \le \|V\| \left( \tfrac{1}{2\gamma} \| u\|^2_{L^2} + \tfrac{\gamma}{2} \| u'\|^2_{L^2}   \right),
\end{split}
\end{equation}
setting $\gamma = h^2/\| V\|$ yields
\begin{equation} \label{norm equiv}
- \tfrac{\| V\|^2}{2h^2} \|u\|^2_{L^2} + \tfrac{h^2}{2} \|u'\|^2_{L^2}   \le q(u,u) \le \tfrac{ \| V\|^2}{2h^2} \|u\|^2_{L^2}  + \tfrac{3h^2}{2} \|u'\|^2_{L^2}.
\end{equation}
 We thus conclude $q$ is semibounded and closed. 
 

By Friedrichs' result \cite[Theorem 2.14]{te}, there is a unique (densely defined) self-adjoint operator $(A, \mathcal{D}_1)$ with 
\begin{equation*}
\begin{gathered}
\mathcal{D}_1 = \{ u \in H^1(\R) : \text{there exists $\tilde{u} \in L^2$ with $q(u,v) = \langle \tilde{u}, v \rangle_{L^2}$ for all $v \in H^1(\R)$} \}, \\
A u  = \tilde{u}.  
\end{gathered}
\end{equation*}
Revisiting the calculation \eqref{prod to form}, we see that for any $u \in \mathcal{D}_1$, $\tilde{u} = -h^2 u'' + uV$ in the distributional sense. Thus $(A, \mathcal{D}_1) \subseteq (P, \mathcal{D}_{\max})$, so $P^* \subseteq A^* = A \subseteq P$. Since we already showed $P \subseteq P^*$ (symmetricity), we conclude $P^* = P$ as desired. \\
\end{proof}

\section{Weighted estimate} \label{measure section}

The purpose of this Section is to prove Theorem \ref{meas Carleman est thm}. As discussed in Section \ref{introduction section}, we do so by means of a positive commutator argument that leverages the energy method.

\begin{proof}[Proof of Theorem \ref{meas Carleman est thm}]
Our starting point is the pointwise energy $F$ given by \eqref{F}. As in the proof of Lemma \ref{self adjointness lemma}, we fix with a continuous representative of $u \in \mathcal{D}$, and fix a representative of $u' \in L^2(\R) \cap L^\infty(\R)$ that has locally bounded variation and $(u')^A = u'$ (thus $F^A = F$). 

Since the measure $V = V(x;h)$ is finite, it can have only countably many point masses $\{x_j\}_j$, and moreover $\sum_{j} |V_j| < \infty$, where $V_j \defeq V(\{x_j\})$. Let us decompose
\begin{equation*}
V = V_c + V_d, \qquad V_d = \sum_j V_j \delta_{x_j},
\end{equation*}
into its discrete and continuous parts. Here, $\delta_{x_j}$ denotes the dirac measure concentrated at $x_j$. A key technical feature of the ensuing calculations is our use of the weight function 
\begin{equation} \label{w}
w = w_\eta(x) \defeq \exp\Big(\int_{-\infty}^x \big[ \frac{1}{E^{1/2}h} |V_c| + (\frac{1}{E^{1/2}h}V_{d, \eta}(x') + (|x'| + 1)^{-1-\delta}) \big]dx' \Big), \qquad \eta > 0,
\end{equation}
where $|V_c|$ denotes the total variation of $V_c$, and
\begin{equation} \label{smoothed discrete part}
V_{d, \eta}(x) \defeq \pi^{-1/2}  \eta^{-1} \sum_j |V_j| e^{-((x-x_j)/\eta)^2}.
\end{equation}
The exponent of $w$ is a continuous function, thus we may compute $dw$ using \eqref{chain rule continuous}. Nearing the end of the argument, we manage to control a term involving $|V_d|$ by sending $\eta \to 0^+$, essentially using that $\pi^{-1/2} \eta^{-1}  \sum_j |V_j|  e^{-((x-x_j)/\eta)^2} \to \sum_j |V_j| \delta_{x_j}$ in the distribution sense. We note also that

 \begin{equation}  \label{sup w}
\sup_{\R} |w_\eta(x)| \le e^{C_1(V,E,h,\delta)}, 
\end{equation}
with $C_1(V, E, h, \delta)$ given by \eqref{C1}.

From \eqref{e:prod} and $u^A= u$, $(u')^A = u'$, we find, in the sense of measures on $\R$,  
\begin{equation*} 
\begin{split}
    dF &= 2 h^2 \real (\overline{u}' d(u')) + 2E \real \left( u \overline{u}' \right) \\
    &= -2 \real (((P - E \pm i\ep)u)\overline{u}') \mp 2\ep \imag \left(u \overline{u}'\right)  + 2\real(u\overline{u}'V),
    \end{split}
\end{equation*}
where, to get the second line, we used that $-h^2 d(u') = Pu - uV$ as Borel measures. Using \eqref{e:prod} again, this time to expand $d(wF)$,
 \begin{equation} \label{meas eta dwF} 
\begin{split}
    d(w F) &= Fdw + w dF \\
    &= |hu'|^2 dw +  E |u|^2dw \\
   &-2  w  \real(((P - E \pm i\ep)u)\overline{u}')  \mp 2\ep w \imag \left(u \overline{u}'\right) + 2\real w (u\overline{u}'V) \\
    &\ge  -2  w  \real(((P - E \pm i\ep)u)\overline{u}')  \mp  2\ep w \imag \left(u \overline{u}'\right) \\
    &+  (|hu'|^2 + E|u|^2)dw - h^{-1}w (E^{1/2} |u|^2 + E^{-1/2}|hu'|^2)(|V_c| +  \sum_{j} |V_j| \delta_{x_j}).
    \end{split}
\end{equation}

By \eqref{chain rule continuous} and \eqref{w},
\begin{equation} \label{d w eta}
dw_\eta = h^{-1}w_\eta ( E^{-1/2} |V_c| + E^{-1/2} \eta^{-1} \pi^{-1/2} \sum_j |V_j| e^{-((x-x_j)/\eta)^2}dx + h(|x| + 1)^{-1 -\delta}dx).
\end{equation}
Plugging \eqref{d w eta} into the fifth line of \eqref{meas eta dwF} implies,
\begin{equation} \label{lwr bd d w eta F}
\begin{split}
 d(w F)&\ge -2  w  \real(((P - E \pm i\ep)u)\overline{u}')  \mp  2\ep w \imag \left(u \overline{u}'\right)  \\
  &+  w (|x| +1)^{-1 -\delta} (E|u|^2  +|hu'|^2)dx \\
 & + \sum_{j} h^{-1}  w(E^{1/2} |u|^2 + E^{-1/2}|hu'|^2)  |V_j|( \eta^{-1} \pi^{-1/2} e^{-((x-x_j)/\eta)^2}dx - \delta_{x_j})
 \end{split}
\end{equation}
Next, we note there exist sequences $\{a^\pm_n\}_{n=1}^\infty$ tending to $\pm \infty$, along which $F(a^\pm_n) = F^R(a^\pm_n) = F^L(a^\pm_n) \to 0$. This is because $F(x) \in L^1(\R)$ and is continuous off of a countable set. So, we integrate both sides of \eqref{lwr bd d w eta F} over $(a^-_n, a^+_n]$ and send $n \to \infty$. By \eqref{ftc}, the left side of \eqref{lwr bd d w eta F} becomes zero. Hence, from \eqref{sup w} and $w \ge 1$,
\begin{equation} \label{meas pre penult est}
\begin{split}
\int & (|x| + 1)^{-1-\delta}(E|u|^2 + |hu'|^2)dx \\
&+ \sum_j h^{-1}  |V_j|  \int  w (E^{1/2} |u|^2 + E^{-1/2}|hu'|^2) (\eta^{-1} \pi^{-1/2} e^{-((x-x_j)/\eta)^2}dx - \delta_{x_j}) \\
 &\le  e^{C_1(V,E,h,\delta)} \big( \int   \frac{1}{\gamma h^2} (|x| + 1)^{1 + \delta}|(P- E \pm i\ep) u|^2 + \gamma (|x| + 1)^{-1-\delta}|hu'|^2 \\
&+ 2\varepsilon \int |uu'| dx \big) , \qquad \gamma, \, h >0. \end{split}
\end{equation}

The goal of the following calculations is to show that the second line of \eqref{meas pre penult est} is nonnegative in the limit as $\eta \to 0^+$. First notice that as $\eta \to 0^+$,
\begin{equation*}
\begin{gathered}
\int_{-\infty}^{x_j} V_{d, \eta}(x') dx'= \pi^{-1/2} \sum_{\ell} |V_\ell| \int_{-\infty}^{(x_j - x_\ell)/\eta} e^{-(x')^2} dx' \to \frac{1}{2}|V_j| +  \sum_{x_\ell < x_j} |V_\ell|, \\
 \int_{-\infty}^{x_j + \eta x} V_{d, \eta}(x') dx'= \pi^{-1/2} \sum_{\ell}  |V_\ell|  \int_{-\infty}^{\frac{x_j - x_\ell}{\eta} + x} e^{-(x')^2} dx' \to \pi^{-1/2} |V_j| \int_{- \infty}^x  e^{-(x')^2} dx' +  \sum_{x_\ell < x_j} |V_\ell|. 
\end{gathered}
\end{equation*}
This implies 
\begin{equation*}
\begin{gathered}
w_\eta(x_j) \to e^{\Gamma(E, h ,j)} \exp\Big( \frac{|V_j|}{2E^{1/2}h} \Big) \\
w_\eta(x_j + \eta x) \to  e^{\Gamma(E, h ,j)}  \exp\Big( \frac{|V_j|}{(\pi E)^{1/2}h} \int_{- \infty}^x  e^{-(x')^2} dx' \Big), 
\end{gathered}
\end{equation*}
where 
\begin{equation*}
e^{\Gamma(E, h ,j)}  \defeq  \exp\Big( \frac{1}{E^{1/2}h} \sum_{x_\ell < x_j} |V_\ell|  + \int_{-\infty}^{x_j} \big[ \frac{1}{E^{1/2}h} |V_c| + (|x'| + 1)^{-1-\delta})dx' \big] \Big).
\end{equation*}
Therefore,
\begin{equation*} 
\begin{split}
h^{-1} |V_j| \int  & w_\eta (E^{1/2} |u|^2 + E^{-1/2}|hu'|^2)\delta_{x_j} \\
&= h^{-1} |V_j| w_\eta (x_j) (E^{1/2} |u(x_j)|^2 + E^{-1/2}|hu'(x_j)|^2) ,  \\
&\to h^{-1} |V_j|  (E^{1/2} |u(x_j)|^2 + E^{-1/2}|hu'(x_j)|^2) e^{\Gamma(E, h ,j)} \exp\Big( \frac{|V_j|}{2E^{1/2}h} \Big),
\end{split}
\end{equation*} 
while 
\begin{equation*} 
\begin{split}
h^{-1} &\eta^{-1} \pi^{-1/2} |V_j| \int   w_\eta (E^{1/2} |u|^2 + E^{-1/2}|hu'|^2) e^{-((x-x_j)/\eta)^2}dx \\
&= h^{-1}  \pi^{-1/2} |V_j| \int w_\eta (x_j + \eta x)  (E^{1/2} |u(x_j + \eta x)|^2 + E^{-1/2}|hu'(x_j + \eta x)|^2)  e^{-x^2}dx ,  \\
&\to   (E^{1/2} |u(x_j)|^2 + E^{-1/2}|hu'(x_j)|^2) e^{\Gamma(E, h ,j)} \\
&\cdot h^{-1}  \pi^{-1/2} |V_j| \int \exp\Big( \frac{|V_j|}{(\pi E)^{1/2}h} \int_{- \infty}^x  e^{-(x')^2} dx'   \Big) e^{-x^2} dx \\
&=  E^{1/2}  (E^{1/2} |u(x_j)|^2 + E^{-1/2}|hu'(x_j)|^2) e^{\Gamma(E, h ,j)}  \big(\exp \Big( \frac{|V_j|}{E^{1/2}h} \Big)    - 1\big). 
\end{split}
\end{equation*} 

In summary, we have shown that the second line of \eqref{meas pre penult est}, upon sending $\eta \to 0^+$, converges to
\begin{equation} \label{crucial nonnegativity}
\sum_j  (E |u(x_j)|^2 + |hu'(x_j)|^2) e^{\Gamma(E, h ,j)} 
 \big( \exp \Big( \frac{|V_j|}{E^{1/2}h} \Big)    - 1 -\frac{|V_j|}{E^{1/2}h}  \exp\Big( \frac{|V_j|}{2E^{1/2}h} \Big) \big) \ge 0.
\end{equation}
The nonnegativity follows from the fact that $e^x - 1 - xe^{x/2} \ge 0$ for all $x \ge 0$.

Returning to \eqref{meas pre penult est}, we fix $\gamma = 2^{-1}e^{-C_1(V,E,h,\delta)}$, so that we may absorb the first term in line three  into the left side, and invoke \eqref{crucial nonnegativity}, implying

\begin{equation} \label{meas penult est}
\begin{split}
\int  (|x| &+ 1)^{-1-\delta}(E|u|^2 + \tfrac{1}{2}|hu'|^2)dx \\
 &\le e^{C_1(V,E,h,\delta)} \Big( \frac{2e^{C_1(V,E,h,\delta)}}{h^2} \int (|x| + 1)^{1 + \delta}|(P- E \pm i\ep) u|^2 dx  +  2\varepsilon \int |uu'| dx\Big), \qquad h > 0.
\end{split}
\end{equation}

For the term in \eqref{meas penult est} having the factor of $\ep$:
\begin{equation}
   2\int |u u'|dx \le \frac{1}{h } \int |u|^2dx +  \frac{1}{h } \int  |h u'|^2dx, \qquad h > 0,  \label{w u u prime} 
  \end{equation}
  and
  \begin{equation} \label{h u prime square 1}
\begin{split}
       \int |hu'|^2dx &= \real \int((P - E \pm i\ep)u)\overline{u}dx + E\int |u|^2 dx + \int |u|^2V   \\
       &\le   \frac{1}{2} \int |(P - E \pm i\ep)u|^2dx + \big(\frac{1}{2} + E\big) \int |u|^2dx + \|V\|  \|u\|_{L^2}  \|u'\|_{L^2}   \\
         &\le  \frac{1}{2} \int |(P - E \pm i\ep)u|^2dx + \big(\frac{1}{2} + E + \frac{\|V\|}{2\gamma h^2} \big) \int |u|^2dx  \\
         &+  \frac{\gamma}{2} \| V\| \int  |h u'|^2dx, \qquad \gamma, \, h > 0.
   \end{split}
\end{equation}
Fixing $\gamma = \|V\|^{-1}$ in \eqref{h u prime square 1},  implies 
\begin{equation} \label{h u prime square 2}
        \int  |hu'|^2 dx \le  \int |(P - E \pm i\ep)u|^2dx + (1 + 2E + \frac{\|V\|^2}{ h^2} )  \int |u|^2dx, \qquad h > 0.
\end{equation}
Now, replace $\int |hu'|^2dx$ in \eqref{w u u prime} by the right side of \eqref{h u prime square 2}. From this, we get a bound for $2\int |uu'|dx$, which we insert into the in the last line of \eqref{meas penult est}. We conclude  
\begin{equation} \label{meas final estimate}
\begin{split}
\int  (|x| &+ 1)^{-1-\delta}(E|u|^2 + \tfrac{1}{2}|hu'|^2)dx \\
 &\le e^{C_1(V,E,h,\delta)}\Big( \big( \frac{2e^{C_1(V,E,h,\delta)}}{h^2} + \frac{1}{h} \big)\int (|x| + 1)^{1 + \delta}|(P- E \pm i\ep) u|^2 dx\\
 &+  \frac{\ep}{h} \big(2 + 2E + \frac{\|V\|^2}{ h^2} \big) \int |u|^2dx  \Big), \qquad \ep \in [0,1], \, h > 0.
\end{split}
\end{equation}

We absorb the last term of \eqref{meas final estimate} into the left side by estimating 
\begin{equation} \label{absorb last error}
\begin{split}
 \varepsilon \|u\|^2_{L^2} &=  \mp \imag \langle (P - E \pm i\ep)u, u  \rangle_{L^2} \\
 &\le \tfrac{\gamma^{-1}}{2} \|  (|x| + 1)^{\frac{1 + \delta}{2}}(P- E \pm i\ep) u\|^2_{L^2} + \tfrac{\gamma}{2}  \|(|x|+1)^{-\frac{1 + \delta}{2}} u \|^2_{L^2}, 
\end{split}
 \end{equation}
 and then choosing $\gamma = Eh^{-1}(2 + 2E + h^{-2} \| V \|^2)^{-1} e^{-C_1(V,E,h,\delta)}$. We then have
\begin{equation} \label{after absorb}
\begin{split}
\int  (|x| &+ 1)^{-1-\delta}(\tfrac{E}{2}|u|^2 + \tfrac{1}{2}|hu'|^2)dx \\
 &\le e^{C_1(V,E,h,\delta)} \Big( \frac{2e^{C_1(V,E,h,\delta)}}{h^2} + \frac{1}{h}  +  \frac{1}{2Eh^2} \big(2 + 2E + \frac{\|V\|^2}{h^2} \big)^2  e^{C_1(V,E,h,\delta)} \Big) \\
& \cdot \int (|x| + 1)^{1 + \delta}|(P- E \pm i\ep) u|^2 dx,  \qquad \ep \in [0,1], \, h > 0.
\end{split}
\end{equation}
 which implies \eqref{meas Carleman est}.\\
\end{proof}

\section{Exterior estimate}  \label{ext est section}
\begin{proof}[Proof of Theorem \ref{ext est meas thm}]
We again start from \eqref{F}, considering the case $\ep \in (0, 1]$ and putting
 \begin{equation*}
 f \defeq (P(h) - E \pm i\ep)^{-1} (|x| + 1)^{-(1+ \delta)/2} u.
 \end{equation*}
This time, we pair $F$ with a much simpler weight $w$. In particular we take $w$ to be the continuous, odd function vanishing on $[-R_0,R_0]$, and obeying
\begin{equation*}
w(x) = 1 - \frac{(1+R_0)^\delta}{(1+x)^{\delta}}, \quad x > R_0 \implies dw = w'(x) = \delta(1+R_0)^{\delta}(1+|x|)^{-1-\delta} \textbf{1}_{\ge R_0}.
\end{equation*}
Note that the same $w$ was used in the proof of \cite[Theorem 2]{dash20}. Since $wV = 0$, we find, proceeding as in \eqref{meas eta dwF},
 \begin{equation} \label{meas ext dwF} 
\begin{split}
    d(wF) = -2w^A   \real(((P - E \pm i\ep)u)\overline{u}')  \mp  2\ep w^A \imag \left(u \overline{u}'\right) + |hu'|^2 dw + E|u|^2dw,
      \end{split}
\end{equation}
and thus
\begin{equation} \label{pre penult meas ext est}
\begin{split}
   \delta(1+R_0)^{\delta} \int_{\R \setminus [-R_0, R_0]} (|x| &+ 1)^{-1-\delta}\big(E|u|^2 + |hu'|^2\big)\\
    &\le  \frac{1}{\gamma h^2}\int_{\R \setminus [-R_0, R_0]} |f |^2 + \gamma \int_{\R \setminus [-R_0, R_0]} (|x| + 1)^{-1-\delta}|hu'|^2  \\
     &+ 2  \varepsilon \int_{\R \setminus [-R_0, R_0]} w |u u'| \qquad \gamma, \, h > 0. 
\end{split}
\end{equation}
Taking $\gamma = 2^{-1} \delta (1 + R_0)^\delta$, we absorb the second term on the right side of \eqref{pre penult meas ext est} into the left side. To handle the term involving $\ep$, we proceed to find
\begin{equation}
   2\int_{\R \setminus [-R_0, R_0]} w|u u'| \le \frac{1}{h} \int_{\R \setminus [-R_0, R_0]} |u|^2 +  \frac{1}{h} \int_{\R \setminus [-R_0, R_0]}  w^2 |h u'|^2, \qquad h > 0, \label{ext w u u prime}
\end{equation}
and
\begin{equation}
    \begin{split} 
       \int_{\R \setminus [-R_0, R_0]} &w^2 |h u'|^2 = 2h^2 \real \int_{\R \setminus [-R_0, R_0]} w w' u' \overline{u} + \real \int_{\R \setminus [-R_0, R_0]} w^2(-h^2u'') \overline{u} \\
       &\le \delta(1 + R_0)^\delta h \int_{\R \setminus [-R_0, R_0]} \big(|u|^2 +w^2  |hu'|^2\big) \\
       &+ \real \int_{\R \setminus [-R_0, R_0]} w^2((P - E \pm i\ep)u)\overline{u}  + E \int_{\R \setminus [-R_0, R_0]} |u|^2 \\
        &\le \frac{1}{2} \int_{\R \setminus [-R_0, R_0]} |f|^2 + \big(\frac{1}{2} + E+ \delta(1 + R_0)^\delta h) \int_{\R \setminus [-R_0, R_0]} |u|^2 \\
        &+ \delta(1 + R_0)^\delta h  \int_{\R \setminus [-R_0, R_0]} w^2 |hu'|^2, \qquad h > 0.
    \end{split} \label{meas ext h u prime square}
 \end{equation}
Putting $h_0 \defeq 2^{-1} \delta^{-1}(1 + R_0)^{-\delta}$ and restricting $h \in (0,h_0]$ in \eqref{meas ext h u prime square} allows us to bound $\int_{\R \setminus [-R_0, R_0]}  w^2 |h u'|^2$ in \eqref{ext w u u prime} by twice the fourth line of \eqref{meas ext h u prime square}. Inserting the resulting estimate for $2\int_{\R \setminus [-R_0, R_0]} w|u \overline{u}'|$ into the right side of \eqref{pre penult meas ext est}  yields, for $\varepsilon \in (0, 1]$ and $h \in (0, h_0]$, 
\begin{equation} \label{penult meas ext est}
\begin{split}
     &\int_{\R \setminus [-R_0, R_0]}  (|x| + 1)^{-1-\delta}|u|^2  \\
     &\le \big( \frac{2}{\delta^2 E (1 + R_0)^{2\delta} h^2} + \frac{1}{ E \delta(1 + R_0)^\delta h} \big) \int_{\R \setminus [-R_0, R_0]} |f|^2 + \ep \frac{3 + 2E}{E \delta(1 + R_0)^\delta h} \int_{\R \setminus [-R_0, R_0]}  |u|^2. \\
     \end{split}
 \end{equation}
 The last term in line two of \eqref{penult meas ext est} may be estimated in manner similar \eqref{absorb last error}, leading to \eqref{ext est meas}. \\
\end{proof}

\section{Uniform resolvent estimate and resonance free strips} \label{unif resolv est section}

In this Section, we prove Theorem \ref{unif resolv est thm} as an application of Corollary \ref{meas resolv est thm}. We are concerned with the self-adjoint operator 
\begin{equation} \label{H}
H \defeq -\partial^2_x + V: \mathcal{D} \to L^2(\R),
\end{equation}
where $V$ remains a finite signed Borel measure, and has support in $[-R_0, R_0]$ for some $R_0> 0$. 

In this situation, $H$ is a \textit{black box Hamiltonian} in the sense of Sj\"ostrand and Zworski \cite{sz91}, as defined in \cite[Definition 4.1]{dz}. More precisely, in our setting this means the following. First, if  $u \in \mathcal{D}$, then $u|_{\mathbb R\setminus [-R_0,R_0]} \in H^2(\mathbb R\setminus [-R_0,R_0])$.  Second, for any $u \in \mathcal{D}$, we have $(Hu)|_{\mathbb R\setminus [-R_0,R_0]}  = - (u|_{\mathbb R\setminus [-R_0,R_0]})''$. Third, any $u \in H^2(\mathbb R)$ which vanishes on a neighborhood of $[-R_0,R_0]$ is also in $\mathcal{D}$. Fourth, $\textbf{1}_{[-R_0,R_0]} (H+i)^{-1}$ is compact on $\mathcal H$; this last condition follows from the fact that $\mathcal{D} \subseteq H^1(\R)$.

Then, by the analytic Fredholm theorem (see \cite[Theorem 4.4]{dz}), we have the following. In $\imag \lambda > 0$, the resolvent $(H - \lambda^2)^{-1}$ is meromorphic $L^2(\R) \to \mathcal{D}$; $\lambda$ is a pole of $(H -\lambda^2)^{-1}$, if and only if $\lambda^2 < 0$ is an eigenvalue of $H$. Furthermore, for $\chi \in C_0^\infty(\mathbb R; [0,1])$ with $\chi =1 $ near $[-R_0, R_0]$, the cutoff resolvent $\chi (H - \lambda^2)^{-1} \chi$ continues meromorphically $L^2(\R) \to \mathcal{D}$ from $\imag \lambda > 0$ to $\C$. The poles of the continuation are known as its \textit{resonances}. 


\begin{proof}[Proof of Theorem \ref{unif resolv est thm}]
Throughout the proof, we use $C(\|V\|, \lambda_0)$ to denote a positive constant which may depend on $\|V \|$ and $\lambda_0$, and whose value may change from line to line, but is always independent of $\lambda$. 

We first show \eqref{unif resolv est} for $k = 0$, $\imag \lambda > 0$, and $|\real \lambda | \ge \lambda_0$.  In this case, let us expand
\begin{equation} \label{expand lambda square}
\begin{split}
\chi (H &- \lambda^2)^{-1} \chi\\
&=  \chi (\partial^2_x + V  - (\real \lambda)^{2} + (\imag \lambda)^2 - i 2\real \lambda \cdot \imag \lambda)^{-1} \chi \\
 &= (\real \lambda)^{-2} \chi ((\real \lambda)^{-2} \partial^2_x + (\real \lambda)^{-2} V  -1  + (\imag \lambda)^2 (\real \lambda)^{-2}  - i 2(\real \lambda)^{-1} \imag \lambda)^{-1} \chi. 
\end{split}
\end{equation}
If $\imag \lambda \ge |\real \lambda|/2$, then by the spectral theorem for self-adjoint operators,
\begin{equation} \label{use spect thm}
\|\chi (H - \lambda^2)^{-1} \chi\|_{L^2 \to L^2} \le |\real \lambda|^{-2} \le \lambda_0^{-1} |\real \lambda|^{-1}, \qquad \imag \lambda \ge |\real \lambda|/2, \, |\real \lambda | \ge \lambda_0.
\end{equation}
If $\imag \lambda < |\real \lambda|/2$, we apply \eqref{exp est meas} to \eqref{expand lambda square} (the notational correspondence is $\delta =1$, $E = 1 -   (\imag \lambda)^2 (\real \lambda)^{-2} \ge 3/4$, $\ep = 2 |\real \lambda|^{-1} \imag \lambda \in (0,1]$, $h =|\real \lambda|^{-1}$, and $V(x, h) = h^2 V$). Therefore,
\begin{equation} \label{L2 L2 bd uhp} 
\|\chi (H - \lambda^2)^{-1} \chi\|_{L^2 \to L^2} \le C(\|V\|,  \lambda_0) |\real \lambda|^{-1}, \qquad  \imag \lambda > 0, \, |\real \lambda | \ge \lambda_0.
\end{equation}

Next, we adapt the proof of \cite[Proposition 2.5]{bgt04} to show
\begin{equation} \label{L2 H1 bd uhp}
\|\chi (H - \lambda^2)^{-1} \chi\|_{L^2 \to H^1} \le C(\|V\|,  \lambda_0), \qquad 0 < \imag \lambda \le 1, \, |\real \lambda | \ge \lambda_0.
\end{equation}
We employ the notation, 
\begin{equation} \label{apply op}
(H - \lambda^2) u = \chi f, \qquad 0 < \imag \lambda \le 1, \, |\real \lambda | \ge \lambda_0,\, f \in L^2(\R), \, u \in \mathcal{D}, 
\end{equation}
and make use of additional cutoffs
\begin{equation} \label{cutoffs}
 \chi_1,\, \chi_2 \in C^\infty_0(\R ; [0,1]), \quad \chi_1 = 1 \text{ on } \supp \chi, \quad \chi_2 = 1 \text{ on } \supp \chi_1. 
\end{equation}
Observe
\begin{equation*}
\| \chi(H - \lambda^2)^{-1} \chi f \|_{H^1} \le \| \chi u \|_{L^2} + \|(\chi u)'\|_{L^2} \le C( \| \chi_2 u \|_{L^2} + \|\chi_1 u'\|_{L^2}),
\end{equation*}
where here and below, $C$ is a positive constant, which may depend on $\|V\|$ and the derivatives of the cutoffs, and which may change between lines, but stays independent of $\lambda$. So, by \eqref{L2 L2 bd uhp} it suffices to show 
\begin{equation} \label{chi1 u prime}
\|\chi_1 u'\|^2_{L^2} \le C \big( (|\real \lambda| +1)^2 \|\chi_2 u\|_{L^2}^2 + \|\chi_2 f\|_{L^2}^2 \big).
\end{equation}

Multiplying \eqref{apply op} by $\chi^2_1 \overline{u}$ and integrating gives
\begin{equation*}
\int \chi_1^2 \chi f \overline{u}dx = \int \chi_1^2 |u'|^2dx + 2 \int \chi_1' \overline{u} \chi_1 u' dx+ \int \chi^2_1 |u|^2 V - \lambda^2 \int \chi^2_1 |u|^2dx, 
\end{equation*}
where \eqref{prod to form} was used, and consequently 
\begin{equation} \label{prepare to absorb}
\begin{split}
\int \chi_1^2 |u'|^2dx &\le \int |\chi_1^2 \chi f \overline{u}|dx  + 2\int |\chi_1' u' \chi_1 \overline{u}| dx  +\| V\| \| \chi_1 u\|_{L^2}  \| (\chi_1 u)'\|_{L^2} \\
&+ (|\real \lambda| + 1)^2 \int \chi^2_1 |u|^2 dx \\
&\le C\big( \|\chi_2 f\|^2_{L^2}  + (|\real \lambda| + 1)^2 \|\chi_2 u\|^2_{L^2} \big) + \frac{1}{2} \int \chi_1^2 |u'|^2dx. 
\end{split}
\end{equation}
Absorbing the last term on the right side into the left side confirms \eqref{chi1 u prime}. 

With \eqref{L2 L2 bd uhp} and  \eqref{L2 H1 bd uhp}, for $0 < \imag \lambda \le 1, \, |\real \lambda | \ge \lambda_0$, and $f \in L^2(\R)$,
\begin{equation*}
\begin{split}
\| \chi (H- \lambda^2)^{-1} \chi f \|_{\mathcal{D}} & \le   \| \chi (H- \lambda^2)^{-1} \chi f \|_{L^2} + \| H\chi (H- \lambda^2)^{-1} \chi f \|_{L^2} \\
&\le  \| \chi (H- \lambda^2)^{-1} \chi f \|_{L^2} + \| [-\partial^2_x, \chi] \chi_1 (H- \lambda^2)^{-1} \chi f \|_{L^2} \\
& + \| \chi ((H - \lambda^2) + \lambda^2) (H- \lambda^2)^{-1} \chi f \|_{L^2} \\
&\le  \| \chi (H- \lambda^2)^{-1} \chi f \|_{L^2} +  \| f \|_{L^2}\\
&+ \| [-\partial^2_x, \chi] \chi_1 (H- \lambda^2)^{-1} \chi f \|_{L^2} + (|\real \lambda| + 1)^2 \| \chi (H- \lambda^2)^{-1} \chi f \|_{L^2} \\
&\le C(\| V\|, \lambda_0) (|\real \lambda| + 1) \| f \|_{L^2}.
\end{split}
\end{equation*}
This implies \eqref{unif resolv est D} for  $0 < \imag \lambda \le 1, \, |\real \lambda | \ge \lambda_0$, and that continued resolvent $L^2(\R) \to \mathcal{D}$ has no poles in $\R \setminus \{0\}$ (since $\lambda_0 > 0$ is arbitrary).

Now, we turn to showing \eqref{unif resolv est D} in strips in the lower half plane. For this, we use a resolvent identity argument due to Vodev \cite[Theorem 1.5]{vo14}, adapted to the non-semiclassical case. It yields holomorphicity of $\chi (H-\lambda^2)^{-1} \chi : L^2(\R) \to \mathcal{D}$ in $|\real \lambda | \ge \lambda_0, \, - \ep_0 < \imag \lambda  <  0$ ($\ep_0 > 0$ sufficiently small), with bounds in these strips of the form \eqref{unif resolv est} and \eqref{unif resolv est D}.

Fix $\tilde{\chi} \in C_0^\infty(\R; [0,1])$ such that $\tilde{\chi} = 1$ near $[-R_0, R_0]$ and $\chi = 1$ near $\supp \tilde{\chi}$. We are going to develop several resolvent identities, and let us work initially with $\lambda, \mu$ such that $\imag \lambda, \ \imag \mu > 0$, $|\real \lambda|,\, |\real \mu| \ge \lambda_0$ (before sending these parameters into the the lower half plane).  By the first resolvent identity, 
\begin{equation*}
\begin{split}
(H - \lambda^2)^{-1} - (H - \mu^2)^{-1} &= (\lambda^2- \mu^2) (H - \lambda^2)^{-1}(H - \mu^2)^{-1} \\
& =(\lambda^2- \mu^2)(H - \lambda^2)^{-1} \tilde{\chi}( 2 - \tilde{\chi})(H - \mu^2)^{-1} \\
&+ (\lambda^2- \mu^2)(H - \lambda^2)^{-1} (1 - \tilde{\chi})^2 (H - \mu^2)^{-1}.
\end{split}
\end{equation*}
As operators on $H^2(\R)$,
\begin{equation*}
\begin{split}
(1- \tilde{\chi})& (- \partial^2_x - \lambda^2) - (H - \lambda^2)(1 - \tilde{\chi}) = [-\partial^2_x,\tilde{\chi}] \implies \\
&(H - \lambda^2)^{-1}(1 - \tilde{\chi}) - (1 - \tilde{\chi})(- \partial^2_x - \lambda^2)^{-1} = (H - \lambda^2)^{-1} [-\partial^2_x, \tilde{\chi}](-\partial^2_x - \lambda^2)^{-1},
\end{split}
\end{equation*}
while as operators on $\mathcal{D}$,
\begin{equation*}
\begin{split}
(-\partial^2_x - \mu^2)&( 1- \tilde{\chi}) - (1- \tilde{\chi})(H - \mu^2)   = [ \partial^2_x, \tilde{\chi}] \implies \\
&( 1-\tilde{\chi})(H - \mu^2)^{-1} - (-\partial^2_x -\mu^2)^{-1}(1- \tilde{\chi}) = (-\partial^2_x -\mu^2)^{-1} [ \partial^2_x, \tilde{\chi}] (H - \mu^2)^{-1}.
\end{split} 
\end{equation*}
 Using $\chi  = 1$ on $\supp \tilde{\chi}$ and the three previous calculations,
\begin{equation} \label{key id for est transfer}
\begin{split}
\chi& (H - \lambda^2)^{-1} \chi - \chi (H - \mu^2)^{-1} \chi  = (\lambda^2 - \mu^2) \chi (H - \lambda^2)^{-1} \chi \tilde{\chi} ( 2 - \tilde{\chi})(H - \mu^2)^{-1} \chi \\
&+ (\lambda^2 - \mu^2) \chi ((1 - \tilde{\chi})(-\partial^2_x - \lambda^2)^{-1} + (H - \lambda^2)^{-1} [- \partial^2_x, \tilde{\chi}](-\partial^2_x - \lambda^2)^{-1})\\
&\cdot ((-\partial^2_x -\mu^2)^{-1}(1- \tilde{\chi}) + (-\partial^2_x -\mu^2)^{-1} [ \partial^2_x,\tilde{\chi}] (H - \mu^2)^{-1}) \chi \\
&= (\lambda^2 - \mu^2) \chi (H - \lambda^2)^{-1} \chi \tilde{\chi} ( 2 -\tilde{\chi})\chi(H - \mu^2)^{-1}\chi \\
&+ (1 - \tilde{\chi} - \chi (H - \lambda^2)^{-1} \chi [\partial^2_x, \tilde{\chi}])\\
&\cdot (\chi (-\partial^2_x -\lambda^2)^{-1} \chi- \chi(-\partial^2_x -\mu^2)^{-1} \chi)\\
& \cdot(1- \tilde{\chi} +  [\partial^2_x, \tilde{\chi}] \chi (H - \mu^2)^{-1}\chi). 
\end{split}
\end{equation}
To get the equality sign in line four of \eqref{key id for est transfer}, we expanded the terms appearing in lines two and three, and repeatedly applied the first resolvent identity to $(\lambda^2 - \mu^2)(-\partial^2_x -\lambda^2)^{-1}(-\partial^2_x -\mu^2)^{-1}$.

Before proceeding to use \eqref{key id for est transfer} to estimate $\|\chi (H - \lambda^2)^{-1} \chi \|_{L^2 \to \mathcal{D}}$ in the lower half plane, we quote a well known estimate for the difference of continued free resolvents (see \cite[Section 5]{vo14} or \cite[Section 3.2]{sh18}):
\begin{equation} \label{line int bd}
\begin{gathered}
\| \chi (-\partial^2 _x-\lambda^2)^{-1} \chi  - \chi (-\partial^2 _x-\mu^2)^{-1} \chi \|_{H^{k_1} \to H^{k_2}}  \le C(\lambda_0) |\lambda - \mu|   \sup_{ \lambda' \in \Gamma_{\lambda, \mu}} |\lambda'|^{k_2 - k_1 -1},\\
k_1 \in \{0, 1\}, \quad k_2 \in  \{0, 1, 2 \}, \quad |\real \lambda|,\, |\real \mu| \ge \lambda_0, \quad  \imag \lambda, \imag \mu \ge -1,
\end{gathered}
\end{equation} 
where $\Gamma_{\lambda, \mu}$ denotes the line segment connecting $\lambda$ and $\mu$.

The identity  \eqref{key id for est transfer} continues to hold after meromorphically continuing both sides ($L^2 \to \mathcal{D}$) to $\lambda,\, \mu \in \C$. Now, fix $\mu \in \R$ with $|\mu| \ge \lambda_0$; Assume $\lambda$ in the lower half plane is not a pole of the continued cutoff resolvent, and obeys $|\lambda - \mu| \le \min(1, \lambda_0/2, \gamma)$, for suitable $0 < \gamma \ll 1$ to be chosen. Then $\|\chi (H - \mu)^{-1} \chi \|_{L^2 \to H^k} \le C(\|V\|, \lambda_0) |\mu|^{k-1}$, $k = 0,\, 1$, $\|\chi (H - \mu)^{-1} \chi \|_{L^2 \to \mathcal{D}} \le C(\|V\|, \lambda_0) |\mu|$, \eqref{key id for est transfer}, and \eqref{line int bd} imply
\begin{equation} \label{absorb resolv norm}
\begin{split}
\|\chi& (H - \lambda^2)^{-1} \chi \|_{L^2 \to \mathcal{D}} \le C(\|V\|, \lambda_0) \Big(| \mu| +|\lambda^2 - \mu^2| |\mu|^{-1} \|\chi (H - \lambda^2)^{-1} \chi \|_{L^2 \to \mathcal{D}} \\
&+ \| (1 - \tilde{\chi}) \chi ((-h^2\partial^2_x -\lambda^2)^{-1} - (-\partial^2_x -\mu^2)^{-1}) \chi \|_{L^2 \to \mathcal{D}} \\
&+ \| \chi ((-h^2\partial^2_x -\lambda^2)^{-1} - (-\partial^2_x -\mu^2)^{-1}) \chi \|_{L^2 \to H^1} \|\chi (H - \lambda^2)^{-1} \chi \|_{L^2 \to \mathcal{D}}\\
&+ |\mu|  \| \chi ((-h^2\partial^2_x -\lambda^2)^{-1} - (-\partial^2_x -\mu^2)^{-1}) \chi\|_{L^2 \to L^2} \|\chi (H - \lambda^2)^{-1} \chi \|_{L^2 \to \mathcal{D}} \Big) \\
&\le C(\|V\|, \lambda_0) \Big( |\mu| + \gamma \|\chi (H - \lambda^2)^{-1} \chi \|_{L^2 \to \mathcal{D}} \Big). \\
\end{split}
\end{equation}
Fixing $\gamma$ small enough (depending on $\lambda_0$) allows us to absorb the term involving $\|\chi (H - \lambda^2)^{-1} \chi \|_{L^2 \to \mathcal{D}}$ on the right side of \eqref{absorb resolv norm} into the left side. This precludes resonances in the region $\imag \lambda \le 0$, $|\lambda - \mu| \le \min(1, \lambda_0/2, \gamma )$, and in this region we have $\|\chi (H - \lambda^2)^{-1} \chi \|_{L^2 \to \mathcal{D}} \le C(\|V\|, \lambda_0)|\mu|$. 

Starting from \eqref{key id for est transfer}, we use the same strategy to show \eqref{unif resolv est} in strips in the lower half plane. Thanks to \eqref{L2 L2 bd uhp}, \eqref{L2 H1 bd uhp}, and \eqref{line int bd}, more negative powers of $|\mu|$ appear while making an estimate similar to \eqref{absorb resolv norm}, since now we need only use operator norms $L^2(\R) \to L^2(\R)$ or $L^2(\R) \to H^1(\R)$.   \\ 
\end{proof}

To conclude this section, we consider the two by two matrix operator
\begin{equation*}
    G \defeq -i \begin{pmatrix} 0 & 1 \\ -H & 0 \end{pmatrix} : \mathcal{D} \oplus L^2(\R) \to L^2(\R) \oplus L^2(\R),
\end{equation*}
which arises naturally from rewriting \eqref{wave eqn} as a first order system. A short computation yields 
\begin{equation} \label{inv G plus lambda}
(G + \lambda)^{-1} = \begin{pmatrix} -\lambda (H - \lambda^2)^{-1} & -i(H - \lambda^2)^{-1} \\ i \lambda^2 (H - \lambda^2)^{-1} + i &  -\lambda (H - \lambda^2)^{-1}  \end{pmatrix} \quad \text{when } \imag \lambda > 0 \text{ and } (H - \lambda^2)^{-1} \text{exists}.
\end{equation}

The following Corollary of Theorem \ref{unif resolv est thm} is essentially well-known, and is an important input to the proof of Theorem \ref{LED thm} in Section \ref{wave decay section}. 
\begin{corollary} \label{matrix op cor}
Let $\chi \in C^\infty_0(\R; [0,1])$ be identically one near $[-R_0, R_0]$. The operator 
\begin{equation} \label{matrix op cutoffs}
\begin{split}
    \chi (G + &\lambda)^{-1} \chi \defeq \\
    & \begin{pmatrix} -\lambda \chi (H - \lambda^2)^{-1} \chi & -i\chi (H - \lambda^2)^{-1}\chi \\ i \lambda^2 \chi (H - \lambda^2)^{-1} \chi + i \chi^2 &  -\lambda \chi (H - \lambda^2)^{-1} \chi  \end{pmatrix} : L^2(\R) \oplus L^2(\R), \to \mathcal{D} \oplus L^2(\R)
\end{split}
\end{equation}
continues meromorphically from $\imag \lambda > 0$ to $\C$, without poles on $\mathbb R \setminus \{0\}$.  For any $\lambda_0 > 0$, there exist $C,\, \lambda_0, \, \ep_0 > 0$ so that
 \begin{equation} \label{matrix op unif bd}
       \|\chi (G + \lambda)^{-1} \chi \|_{H^1(\R) \oplus L^2(\R) \to H^1(\R) \oplus L^2(\R)} \le C, \qquad |\real \lambda| \ge \lambda_0, \, |\imag \lambda| \le \ep_0.
 \end{equation}
 
 If $\chi(G + \lambda)^{-1}\chi$ has a pole at $\lambda = 0$, it is a simple pole. More precisely, if $w_0 \in H^1(\mathbb R)$ and $w_1 \in L^2(\mathbb R)$, then 
\begin{equation}\label{e:residuecomp}
\lim_{\lambda \to 0} \lambda \chi (G+\lambda)^{-1} \chi \begin{pmatrix} w_0 \\ w_1  \end{pmatrix} =  \begin{pmatrix}  -i \lim_{\lambda \to 0} \lambda \chi  (H - \lambda^2)^{-1}\chi w_1\\ 0   \end{pmatrix} = \begin{pmatrix} \langle \chi u_0, w_1 \rangle_{L^2} \chi u_0\\ 0  \end{pmatrix}.
\end{equation}
for some real valued $u_0 \in H^1_{\loc}(\R) \cap H^2_{\loc}(\R \setminus [-R_0, R_0])$ with $Hu_0 = 0 $ in the sense of distributions.

\end{corollary}

\begin{proof}
By the blackbox formalism (see \cite[Definition 4.1 and Theorem 4.4]{dz}) and Theorem \ref{unif resolv est thm}, $\chi (H - \lambda^2)^{-1} \chi : L^2(\R) \to \mathcal{D}$ continues meromorphically from $\imag \lambda >0$ to $\mathbb{C}$, and has no poles in $\R \setminus \{0\}$. This implies that each entry of \eqref{matrix op cutoffs} continues meromorphically as an operator between the appropriate spaces, again without poles in $\R \setminus \{0\}$.

With \eqref{unif resolv est} already in hand, to establish \eqref{matrix op unif bd}, we need to show for any $\lambda_0 > 0$, there exist $C, \, \ep_0 > 0$ so that 
\begin{gather}
\|\lambda^2 \chi (H - \lambda^2)^{-1} \chi +  \chi^2\|_{H^1 \to L^2} = \|\chi H (H - \lambda^2)^{-1} \chi\|_{H^1 \to L^2} \le C , \label{two one entry} \\
\| \lambda \chi (H - \lambda^2)^{-1} \chi \|_{H^1 \to H^1} \le C, \label{one one entry}
\end{gather}
for $ |\real \lambda| \ge \lambda_0$ and $|\imag \lambda| \le \ep_0$. First, we first prove \eqref{two one entry} for $|\real \lambda| \ge \lambda_0$ and $0 < \imag \lambda \le \ep_0$, and then handle the remaining cases.

Let us use the notation
\begin{equation} \label{f in H1}
u = (H - \lambda^2)^{-1} \chi f \in \mathcal{D}, \quad f \in H^1(\R), \quad |\real \lambda | \ge \lambda_0 \text{ and } \imag \lambda > 0.
\end{equation}
Let $\chi_1 \in C_0^\infty(\R)$ with $\chi_1 = 1$ near $\supp \chi$. As we showed in the proof of Lemma \ref{self adjointness lemma}, the form domain of $(H ,\mathcal{D})$ is $H^1(\R)$, so there exists a sequence $f_k \in \mathcal{D}$ converging to $f$ in $H^1(\R)$, and corresponding functions $u_k \defeq (H - \lambda^2)^{-1} \chi f_k$ converging to $u$ in $(\mathcal{D}, \| \cdot \|_{\mathcal{D}})$. Since $Hu_k = (H -\lambda^2)^{-1} \chi_1 H \chi f_k$,
\begin{equation} \label{chi H u}
\begin{split}
\| \chi H u \|_{L^2} = \lim_{k \to \infty} \| \chi H u_k \|_{L^2} \le \lim_{k \to \infty} \|  \chi_1  (H -\lambda^2)^{-1} \chi_1 H \chi f_k\|_{L^2}.
\end{split}
\end{equation}
Furthermore, by \eqref{prod to form}, there exists $C(\|V \|) > 0$ depending on $\| V\|$ so that for any $v \in L^2(\R)$, 
\begin{equation*}
\begin{split} 
|\langle \chi_1  (H -\lambda^2)^{-1} \chi_1 H \chi f_k, v \rangle_{L^2}| &= |\langle  H \chi f_k, \chi_1  (H -(-\overline{\lambda})^2)^{-1} \chi_1v \rangle_{L^2} \\
&\le C(\|V\|) \| \chi f_k \|_{H^1} \| \chi_1  (H -(-\overline{\lambda})^2)^{-1} \chi_1v\|_{H^1}. 
\end{split} 
\end{equation*}
Since \eqref{unif resolv est} gives $\| \chi_1  (H -(-\overline{\lambda})^2)^{-1} \chi_1v\|_{H^1} \le C(\|V \|, \lambda_0) \| v\|_{L^2}$ (provided $\imag \lambda$ is small), we conclude $\|\chi_1  (H -\lambda^2)^{-1} \chi_1 H \chi f_k\|_{L^2} \le C(\|V \|, \lambda_0) \| \chi f_k \|_{H^1}$. Returning to \eqref{chi H u}, we now find 
\begin{equation*} 
\| \chi H u \|_{L^2} \le C(\|V \|, \lambda_0) \lim_{k \to \infty}  \| \chi f_k \|_{H^1} \le C(\|V \|, \lambda_0)  \| f \|_{H^1}, \qquad |\real \lambda| \ge \lambda_0, \, 0 < \imag \lambda \le \ep_0.
\end{equation*}
In turn,
\begin{equation} \label{lambda square unif bd}
\begin{gathered}
\|\lambda^2 \chi (H - \lambda^2)^{-1} \chi\|_{H^1 \to L^2} = \| \chi H (H- \lambda^2)^{-1} \chi - \chi^2\|_{H^1 \to L^2}  \le C(\|V\|, \lambda_0), \\ 
|\real \lambda| \ge \lambda_0 \text{ and } 0 < \imag \lambda \le \ep_0.
\end{gathered}
\end{equation}
which is \eqref{two one entry} for $|\real \lambda| \ge \lambda_0$ and $0 < \imag \lambda \le \ep_0$.

Next, with $u$ as in \eqref{f in H1}, we slightly modify the method of estimation in \eqref{prepare to absorb}, this time finding 
\begin{equation*}
\int \chi_1^2 |u'|^2dx  \le C(\|V\|, \lambda_0) \big( |\real \lambda| + 1)^{-2} \|\chi_2 f\|^2_{L^2}  + (|\real \lambda| + 1)^2 \|\chi_2 u\|^2_{L^2} \big),
\end{equation*}
and where we recall from \eqref{cutoffs} that $\chi_2 = 1$ on $\supp \chi_1$. Combining this with \eqref{lambda square unif bd} establishes \eqref{one one entry} when $|\real \lambda| \ge \lambda_0$ and $0 < \imag \lambda \le \ep_0$

To show \eqref{one one entry} and \eqref{two one entry} hold for $|\real \lambda| \ge \lambda_0$ and $-\ep_0 < \imag \lambda < 0$, we revisit \eqref{key id for est transfer} and multiply by the appropriate power of $\lambda$. We then perform an estimate similar to \eqref{absorb resolv norm}. As needed, we invoke \eqref{line int bd} and $\|\mu^2 \chi (H - \mu^2)^{-1} \chi \|_{H^1 \to L^2} ,\, \|\mu \chi (H - \mu^2)^{-1} \chi\|_{H^1 \to H^1} \le C$ ($\mu \in \R, |\mu| \ge \lambda_0$).  

Finally, to show \eqref{e:residuecomp}, we proceed as in the proof of \cite[Theorem 2.7]{dz}. We omit the details, but take care to note that this argument does require that for each $x_0 \in \R \setminus [-R_0, R_0]$ and $a, b \in \C$, there is a unique solution $f$ to $H f = 0$ satisfying $f(x_0) = a$ and $f'(x_0) = b$; Even though $V$ is only a measure in our setting, such well-posedness still holds for the initial value problem, see \cite[Theorem 3.1]{ecte13}. (In general it is not necessary to prescribe the initial conditions outside the support of the measure, but this is sufficient for our purpose). The result is that near $\lambda  = 0$,
\begin{equation*} 
\chi (H - \lambda^2)^{-1} \chi w_1 =  \frac i { \lambda} \langle \chi u_0, w_1 \rangle_{L^2} \chi u_0+ A(\lambda)w_1, \qquad w_1 \in L^2(\R),
\end{equation*}
where $A(\lambda) : L^2(\R) \to \mathcal{D}$ is holomorphic near zero, and for some $u_0 \in H^1_{\loc}(\R) \cap H^2_{\loc}(\R \setminus [-R_0, R_0])$ with $Hu_0 = 0 $ in the sense of distributions. Hence we have \eqref{e:residuecomp}. \\
\end{proof}

\section{wave decay} \label{wave decay section}
In this Section, we combine Corollary \ref{matrix op cor}  with an argument similar to those appearing in \cite[Section 3]{vo99} and \cite[Section 4]{dash22}. We establish exponential local energy decay, modulo negative eigenvalues and a possible zero resonance, for solutions of the wave equation \eqref{wave eqn}.

First, we represent the solution to \eqref{wave eqn} via the spectral theorem of for self-adjoint operators. Additionally, we use that the proof of Lemma \ref{self adjointness lemma} shows the form domain of $(H, \mathcal{D})$ (i.e., the domain of $|H|^{1/2}$) is $H^1(\R)$. Thus, given initial conditions $w_0 \in \mathcal{D}$, $w_1 \in H^1(\R)$, the unique function $w \in C^2((0,\infty), \mathcal{H})$ with $w(0) = w_0$, $\partial_tw(0) = w_1$, $w(t) \in \mathcal{D}(H)$ for all $t > 0$, and $\partial^2_t w(t) +Hw(t)=0$, is
 \begin{equation} \label{soln spectral thm}
 \begin{gathered}
w(t) = w(\cdot, t) = \mathbf{1}_{\ge 0}(H) w(\cdot, t) + \mathbf{1}_{< 0}(H) w(\cdot, t), \\
 \mathbf{1}_{\ge 0}(H) w(\cdot, t) = \mathbf{1}_{\ge 0}(H) \big( \cos(t |H|^{1/2}) w_0 + \frac{\sin(t|H|^{1/2})}{|H|^{1/2}} w_1 \big), \\
  \mathbf{1}_{< 0}(H) w(\cdot, t) =   \mathbf{1}_{< 0}(H) \big( \cos( it |H|^{1/2}) w_0 + \frac{\sin(it|H|^{1/2})}{i|H|^{1/2}} w_1 \big) .\\
\end{gathered}
\end{equation}


\begin{theorem} \label{LED thm}
 Suppose $w_0 \in \mathcal{D}, \, w_1 \in H^1(\R)$, and $\supp w_0,\, \supp w_1 \subseteq (-R, R)$ for some $R> 0$. Let $w(t)$ be given by \eqref{soln spectral thm}. For any   $R_1 > 0$, there exist $C, c > 0$ so that 
\begin{equation} \label{LED}
\begin{gathered}
 \|  \mathbf{1}_{\ge 0}(H) w(\cdot, t) -  w_\infty(x) \|_{H^1(-R_1,R_1)} + 
    \|\partial_t  \mathbf{1}_{\ge 0}(H) w(\cdot, t)  \|_{L^2(-R_1,R_1)} \\
   \le C e^{-c t} (\| w_0\|_{H^1(\R)} + \|w_1\|_{L^2(\R)}), \qquad t > 0.
   \end{gathered}
\end{equation}
Furthermore, if $\chi \in C^\infty_0(\R ; [0,1])$ is identically one near $[-R_1, R_1] \cup [-R_0, R_0] \cup [-R, R]$, then the function $w_\infty(x)$ may be written as
   \begin{equation} \label{w infty}
       w_\infty(x) \defeq  \chi(x) u_0(x) \int_{\R} \chi  u_0 w_1,
   \end{equation}
for some real valued $u_0 \in H^1_{\loc}(\R) \cap H^2_{\loc}(\R \setminus [-R_0, R_0])$ with $Hu_0 = 0 $ in the sense of distributions (and if the continued operator \eqref{matrix op cutoffs} does not have a pole at $\lambda = 0$, we may take $u_0 \equiv 0$.)
\end{theorem}

\begin{proof}

Choose $\chi \in C_0^\infty(\mathbb R; [0,1])$ such that $\chi =1$ near $[-R_1, R_1] \cup [-R, R] \cup [R_0, R_0]$ (where as before $\supp V  \subseteq [-R_0, R_0]$). From Corollary \ref{matrix op cor}, for any $\lambda_0 > 0$, there exist $C$, $\varepsilon_0 > 0$ such that
\begin{equation} \label{matrix op unif bd from cor} 
\|\chi(G + \lambda)^{-1}\chi f\| \le C\|f\|,
\end{equation}
whenever $|\real \lambda | \ge \lambda_0$ and $|\imag \lambda| \le \varepsilon_0$, where here and for the rest of this Section, all norms are $H^1(\R) \oplus L^2(\R)$ unless otherwise specified.

We have
\begin{equation*}
\begin{gathered}
\mathbf{1}_{\ge 0}(H)w(t) = \mathbf{1}_{\ge 0}(H)\big( \cos(t |H|^{1/2}) w_0 + \sin(t|H|^{1/2})|H|^{-1/2} w_1\big),\\
\partial_t \mathbf{1}_{\ge 0}(H) w(t) = \mathbf{1}_{\ge 0}(H) \big( -\sin(t|H|^{1/2}) |H|^{1/2} w_0 + \cos(t|H|^{1/2})w_1 \big), \\
\partial_t^2 \mathbf{1}_{\ge 0}(H) w(t) = - H \mathbf{1}_{\ge 0}(H) w(t).
\end{gathered}
\end{equation*}
Consequently, after defining
\[
 f \defeq \left( \begin{array}{c}w_0\\w_1 
\end{array} \right), \qquad U(t) f \defeq \left( \begin{array}{c} 
\mathbf{1}_{\ge 0}(H) w(t) \\
\partial_t \mathbf{1}_{\ge 0}(H) w(t)\end{array}\right),
\]
we have
\begin{equation}\label{e:utfbound}
\|U(t) f \| \le C(1 + |t|)\|f\|, \qquad \partial_t  U(t) f =i G U(t) f, \qquad U(t)U(s) f =  U(t+s)f,
\end{equation}
for all real $t$ and $s$, and for some $C>0$ independent of $t$ and $f$. 

Take $\varphi \in C^\infty(\mathbb R; [0,1])$ which is $0$ near $(-\infty,1]$ and $1$ near $[2,\infty)$ and put
\begin{equation*} 
 W(t) f \defeq \varphi(t) U(t) f =   \int_{\imag \lambda = \varepsilon}  e^{-it\lambda} \check W(\lambda) \, d \lambda, \qquad \check W(\lambda) \defeq \frac 1 {2\pi} \int_{\R} e^{is\lambda}W(s)fds, \qquad \ep > 0.
\end{equation*}
We compute $\partial_t W(t)f = \varphi'(t) U(t)f + i G W(t)f$, and therefore find
\begin{equation} \label{before deform contour}
 W(t) f = \int_{\imag \lambda = \varepsilon}  e^{-it\lambda} (G + \lambda)^{-1}(i\varphi' U f)\check{~}(\lambda)\,d \lambda, \qquad \ep >0.
\end{equation}

Since $w_0$, $w_1$ and $V$ have compact support, finite propagation speeds holds for the solution \eqref{soln spectral thm}. Therefore, increasing $R > 0$ if necessary, we have that, $x \mapsto U(t)f$ is supported in $(-R,R)$ for all $t \in [0,2]$. By continuity of integration, the same is true of $x \mapsto (i\varphi' U f)\check{~}(\lambda)$ for every $\lambda$. Hence $ \lambda \mapsto (i\varphi' U f)\check{~}(\lambda)$ is entire and rapidly decaying as  $|\real \lambda| \to \infty$ with $|\imag \lambda|$ remaining bounded and further $(i\varphi' U f)\check{~}(\lambda) = \chi(i\varphi' U f)\check{~}(\lambda)$. 

By \eqref{matrix op unif bd from cor}, there exists $\varepsilon > 0$ small enough so that, within the strip $|\imag \lambda| < 2\varepsilon$, either $\chi (G+\lambda)^{-1}\chi$ has no poles, or just a pole at $\lambda = 0$. Deforming the contour in \eqref{before deform contour}, by the residue theorem, we find
\begin{equation*}
\begin{split}
\chi  W&(t)f \\
&= -2\pi i \Res{}_{\lambda = 0}  (e^{-it\lambda} \chi (G + \lambda)^{-1}\chi (i\varphi' U f)\check{~}(\lambda)) + \int_{\imag \lambda = -\varepsilon}  e^{-it\lambda} \chi (G + \lambda)^{-1}\chi (i\varphi' U f)\check{~}(\lambda) d \lambda.  \\
&= \lim_{\lambda \to 0} \lambda \chi (G+\lambda)^{-1} \chi \int_{\R} \varphi'(s) U(s) f\,ds + \int_{\imag \lambda = -\varepsilon}  e^{-it\lambda} \chi (G + \lambda)^{-1}\chi (i\varphi' U f)\check{~}(\lambda) d \lambda. 
\end{split}
\end{equation*}
To simplify this, use \eqref{e:residuecomp} and put
\[
 W_1(t)f \defeq \int_{-\infty}^\infty e^{-it\lambda} \chi (G + \lambda - i \varepsilon)^{-1} \chi (i\varphi' U f)\check{~}(\lambda - i \varepsilon)\, d \lambda,
\]
to obtain
\[
\chi W(t) f = \begin{pmatrix} \chi u_0 \int_{\R} \int_0^2  \chi(x) u_0(x) \varphi'(s) \partial_sw(s,x)ds dx \\ 0 \end{pmatrix} + e^{-\varepsilon t}  W_1(t)f. 
\]
To simplify the first term, we integrate by parts in $s$, using $\varphi' = -(1-\varphi)'$, to obtain
\[
\int_{\R} \int_0^2   \chi(x) u_0(x) \varphi'(s) \partial_sw(s,x)\,ds\,dx =\int_{\R} \chi u_0 w_1 +\int_{\R} \int_0^2  \chi(x) (1-\varphi(s)) \partial_s^2w(s,x)\,ds\,dx. 
\]
Now observe that $\partial_s^2 w = - H w$, $\chi u_0 \in \mathcal{D}$, so
\begin{equation*}
\langle \chi u_0, H w(s) \rangle_{L^2} = \langle H \chi u_0, w(s) \rangle_{L^2} = \langle ([H, \chi] + H )u_0, w(s) \rangle_{L^2} = 0, \qquad  \text{ for } s \in [0,2],
\end{equation*}
the last equality following from $\chi = 1$ near $[-R, R]$ and $\supp w(s) \subseteq (-R, R)$ for $s \in [0,2]$). Thus
\[
\chi W(t) f =  \begin{pmatrix} \langle \chi u_0, w_1 \rangle_{L^2} \chi u_0  \\ 0  \end{pmatrix}+ e^{-\varepsilon t} W_1(t)f.
\]

It now suffices to show that
\[
\| W_1(t)f\| \le C e^{\varepsilon t/2}\|f\|.
\]
To prove this, we first use Plancherel's theorem, along with the fact that by \eqref{e:utfbound}, the operator norm  $\|U(t)\|_{H^1(\R) \oplus L^2(\R) \to H^1(\R) \oplus L^2(\R)}$ is uniformly bounded for all $t \in \mathbb R$, as well as the fact that by Corollary \ref{matrix op cor}, the operator norm $\|\chi (G+\lambda-i\varepsilon)^{-1} \chi\|_{H^1(\R) \oplus L^2(\R) \to H^1(\R) \oplus L^2(\R)}$ is uniformly bounded for all $\lambda \in \R$, to obtain
\begin{equation}\label{e:planch}\begin{split}
\int  \| W_1(t)f\|^2 \, dt &= C \int \|\chi (G + \lambda - i \varepsilon)^{-1} \chi (\varphi' U f)\check{~}(\lambda - i \varepsilon)\|^2 \, d \lambda \\
&\le C_\ep \int \|(\varphi' U f)\check{~}(\lambda - i \varepsilon)\|^2 \, d \lambda \\ &= C_\ep \int e^{2\varepsilon t}  \|\varphi'(t) U(t) f\|^2 \, d t \le C_\ep \|f\|^2.
\end{split}\end{equation}

Next, let $\tilde{\chi} \in C^\infty_0(\R ; [0,1])$ with $\tilde{\chi} = 1$ on $\supp \chi$. Observe that
\begin{equation} \label{commutator id}
G \chi (G + \lambda)^{-1} \chi = [G, \chi]  \tilde{\chi} (G + \lambda + i\ep)^{-1} \tilde{\chi} \chi  - \lambda \chi (G + \lambda)^{-1} \chi  + \chi^2
\end{equation}
holds initially for $\imag \lambda \gg 1$ by \eqref{inv G plus lambda}, and continues meromorphically to $\C$ by \eqref{matrix op cutoffs}. In particular, decreasing $\ep$ if necessary, \eqref{matrix op unif bd} implies that \eqref{commutator id} holds for everywhere in $|\imag \lambda | < 2\ep$, except possibly at $\lambda = 0$. Therefore, setting,
\begin{equation*}
 \widetilde W_1(t)f \defeq \int_{-\infty}^\infty e^{-it\lambda} \tilde{\chi} (G + \lambda - i \varepsilon)^{-1} \tilde{\chi} (i\varphi' U f)\check{~}(\lambda - i \varepsilon)\, d \lambda,
\end{equation*}
we have
\begin{equation*}
 (\partial_t - i G) W_1(t)f = - i [G,\chi] \widetilde W_1(t)f + \varepsilon W_1(t)f - i  \int e^{-it\lambda }(i\varphi' U f)\check{~}(\lambda - i \varepsilon)\, d \lambda \qefed W_2(t)f.
\end{equation*}
Integrating both sides of 
\begin{equation*}
\begin{split}
 \partial_s( U(t-s) W_1(s)f)  &= -i GU(t-s) W_1(s) f + U(t-s)\big(iGW_1(s)f + W_2(s) f\big) \\
 &= U(t-s) W_2(s)f
 \end{split}
\end{equation*}
 from $s=0$ to $s=t$ gives
\begin{equation*}
  W_1(t)f  = U(t) W_1(0)f  + U(t)\int_0^t U(-s) W_2(s)f\,ds.
  \end{equation*}
Thus
\begin{equation*}
\begin{split}
 \| W_1(t) f\| &\le C(1 + t)\big(\|f\| +  \int_0^t (1 + s)\| W_2(s)f\| ds \big) \\
 &\le C(1 + t)\big( \|f\| +  \big(\frac{t^3}{3} + t^2 + t\big)^{1/2} \Big(\int_0^t \| W_2(s)f\|^2 ds\Big)^{1/2} \big).
\end{split}
\end{equation*}
Now check that, since $\|[G,\chi] \widetilde W_1(t)f\| \le C\|\widetilde W_1(t)f\|$, calculating as in \eqref{e:planch}, we obtain \\ $\int \|W_2(s)f\|^2\,ds \le C \|f\|^2$, and hence
\[
\|W_1(t)f\| \le C (1 + t^{5/2}) \|f\|
\]
as desired.\\
\end{proof}

\medskip

\noindent{\textsc{Acknowledgements:}} It is a pleasure to thank Kiril Datchev for helpful discussions. We also wish to thank the anonymous referees, whose valuable comments helped improve this article.

\medskip

\noindent{\textsc{Declarations, Funding and/or Conflicts of interests/Competing interests:}} JS gratefully acknowledges support from ARC DP180100589, NSF DMS 2204322, and from a 2023 Fulbright Future Scholarship funded by the Kinghorn Foundation and hosted by University of Melbourne. On behalf of all authors, the corresponding author states that there is no conflict of interest.\\

\noindent{\textsc{Data Availability Statement:}} Data sharing not applicable to this article as no datasets were generated or analyzed in this study.

\end{document}